\documentclass[12pt]{amsart}
\usepackage{geometry}
\usepackage{amsmath}
\usepackage{amsthm}
\usepackage{amsfonts}
\usepackage{amssymb}
\usepackage{graphicx} 
\usepackage{xcolor}
\usepackage{stmaryrd}

\usepackage{pdfpages}

\newtheorem{theorem}{Theorem}

\newtheorem{proposition}[theorem]{Proposition}

\newtheorem{remark}[theorem]{Remark}
\newtheorem{corollary}[theorem]{Corollary}

\title{Tate-Shafarevich results for quartic twists in characteristic $2$}
\author{Herivelto Borges \\ João Paulo Guardieiro \\ Cecília Salgado \\ Jaap Top}
\address{Instituto de Ciências  Matemáticas e Computação, Av. Trabalhador São Carlense 400, 13566-590 São Carlos, Brazil.}
\address{Bernoulli Institute,  Nijenborgh 9,
9747 AG~Groningen, the Netherlands.}
\email{hborges@icmc.usp.br \\joao.guardieiro@usp.br \\ c.salgado@rug.nl\\ j.top@rug.nl}
\begin{document}

\maketitle

\begin{center}
    \it Dedicated to Sudhir Ghorpade, to whom the last
    author owes many pleasant memories, often at CIRM in Luminy and recently in Hyderabad as well. 
\end{center}

\section{Introduction}
One of the early applications of Tate's famous
``proof of the Tate conjecture in the case of abelian varieties over finite fields'' \cite{Tate66}, is a less than 4 pages note
\cite{TateSha67} by Tate and Shafarevich. The note explains that elliptic curves $E/\mathbb{F}_q(t)$ exist with arbitrarily high rank of
the group of rational points $E(\mathbb{F}_q(t))$, for any finite field $\mathbb{F}_q$ of characteristic $\neq 2$. The proof uses quadratic
twists over $\mathbb{F}_q(t)$, hence function
fields of hyperelliptic curves over $\mathbb{F}_q$. The analogous result in characteristic $2$
was added by Elkies \cite{Elkies}, who used this
case to construct (Mordell-Weil) lattices with
very good sphere-packing densities. The exposition in \cite[Ch.~13]{ShiodaSchuettBook}
(in particular \S~13.3.1) includes an overview
of these Tate-Shafarevich results.

In \cite{BTT} a variant of the approach is presented, using quartic and sextic twists instead of the quadratic ones. This means that
the hyperelliptic curves in the original results
are replaced by cyclic degree $4$ or $6$ covers
of $\mathbb{P}^1$ (and to allow
quartic/sextic twists, only elliptic curves $E$
with $\text{Aut}(E)\neq \{\pm 1\}$ are considered). The assumption used in \cite{BTT} is that the
characteristic $p$ in which the elliptic curves are considered is either $p\equiv 3\bmod 4$ or $p\equiv 5\bmod 6$ (to assure that $E$ is supersingular). At least in the specific example
of $E\colon y^2=x^3+t^{p+1}+1$ this offers
an alternative proof of results by Shioda in
\cite{Shioda91}.

The aim of the present paper is to extend the
methods in \cite{BTT} to the case of quartic
twists in characteristic $2$.
Specifically, start (characteristic $2$) with $E\colon y^2+y=x^3+x$
and its order $4$ automorphism
$\iota(x,y)=(x+1,y+x)$. Take a curve $C$
admitting an automorphism $\sigma$ of order~$4$.
Then $(E\times C)/\langle \iota\times\sigma\rangle\to C/\langle \sigma\rangle$ defines an elliptic surface.
We describe this surface and we present conditions allowing us to find the Mordell-Weil
rank of the given elliptic surface. In particular
this gives rise to (new) examples over
$\mathbb{F}_{2^n}(t)$ of arbitrarily high rank.

\section{Preliminaries}\label{prelim}
This section reviews properties of a specific (supersingular) elliptic curve
defined over $\mathbb{F}_2$. Moreover the special case ``characteristic $2$'' of a
result by Artin and Schreier \cite{AS} is recalled, describing all field
extensions $L\supset K$ of degree $p^2$ in characteristic $p>0$ that are
Galois with a cyclic Galois group. We also cite a consequence of Tate's
result \cite{Tate66} on endomorphisms of abelian varieties over finite fields
important in the presented approach. Finally, our rank formula is presented.
Its proof
and various examples are described in the remaining sections.

The elliptic curve $E/\mathbb{F}_2\colon y^2+y=x^3+x$ admits the
automorphism $\iota$ given by
$\iota(x,y)=(x+1,y+x)$. One verifies $\iota^2=[-1]$, which implies that $\iota$ has order $4$. It is related to the Frobenius endomorphism $F$ on $E$
given by $F(x,y)=(x^2,y^2)$ via
$F=[-1]+\iota$. As a consequence,
\[
\#E(\mathbb{F}_{2^{2n}})=\deg([1]-F^{2n})=\left\{\begin{array}{ll} 2^{2n}+1 &\text{if}\; n\equiv 1\bmod2;\\
 (2^n -1)^2 &\text{if}\; n\equiv 0\bmod 4;\\ 
(2^n+1)^2 &\text{if}\; n\equiv 2\bmod 4.\end{array}\right.
\]
In particular this shows that $E/\mathbb{F}_{2^{2n}}$ is maximal precisely when
$n\equiv 2\bmod 4$, respectively minimal if and only if $n\equiv 0\bmod 4$. Recall that
a (smooth, complete, geometrically irreducible) curve $C/\mathbb{F}_{q^2}$ of
genus $g$ is maximal iff
$\#C(\mathbb{F}_{q^2})=q^2+1+2gq$, resp. minimal iff 
$\#C(\mathbb{F}_{q^2})=q^2+1-2gq$ (compare, e.g., \cite[Def.~5.3.2]{stichtenoth}).
% \begin{definition}
%     Let $C/\mathbb{K}$ be a smooth projective curve. We define the \textbf{isomorphism group} of $C$, Isom($C$), as the group of $\overline{\mathbb{K}}$-isomorphisms that maps $C$ on itself. The subgroup of isomorphisms of $C$ defined over $\mathbb{K}$ will be denoted by $\textrm{Isom}_{\mathbb{K}}(C)$.
% \end{definition}

% \begin{remark}
%     The group Isom($C$) is essentially the same as the automorphism group of the curve $C$, but it may be interesting to distinguish the two terminologies since, on the theory of elliptic curves, an automorphism also fixes the base-point of the curve.
% \end{remark}

% \begin{definition}
%     A \textbf{twist} of a curve $C/\mathbb{K}$ is any curve $C'/\mathbb{K}$ that is isomorphic to $C$ over $\overline{\mathbb{K}}$. We will say that two twists are equivalent if they are isomorphic over $\mathbb{K}$. Therefore, we can consider the set of twists of $C$ modulo $\mathbb{K}$-isomorphism, denoted by Twist($C/\mathbb{K}$).
% \end{definition}

% \begin{theorem}\cite[Theorem X.2.2]{silverman}
%     There is a bijection between Twist($C/\mathbb{K}$) and the cohomology set $H^1(G_{\overline{\mathbb{K}},\mathbb{K}},\textrm{Isom}(C))$, where $G_{\overline{\mathbb{K}},\mathbb{K}}$ is the Galois group of the field extension $\overline{\mathbb{K}}/\mathbb{K}$.
% \end{theorem}

The results in the present paper use nonconstant morphisms $C\to X$ defined over $\mathbb{F}_q$ where $C,X$ are curves and $q$ is a power of $2$. This is done assuming that the corresponding 
extension $\mathbb{F}_q(C)\supset\mathbb{F}_q(X)$ is Galois
with cyclic Galois group of order $4$. In other words, $C$ admits
an automorphism $\sigma$ of order $4$ and $X=C/\langle\sigma\rangle$ (and the morphism $C\to X$ is the quotient map). As a special case of \cite[Section~2]{AS}, Artin and Schreier describe all such extensions: take $A\in \mathbb{F}_q(X)\setminus \wp(\mathbb{F}_q(X))$ where $\wp\colon f\mapsto f^2+f$.   \label{ASform}
Take any $B\in \mathbb{F}_q(X)$. Then
$\mathbb{F}_q(C)=\mathbb{F}_q(X)[r,s]$ with
$r^2+r=A$ and $s^2+s=rA+B$. In this form, the automorphism $\sigma$ is determined by 
\begin{equation}\label{aut-sigma}\sigma(s)=s+r,\quad \sigma(r)=r+1.\end{equation}
One has the `intermediate' curve $H:=C/\langle \sigma^2\rangle$
with function field $\mathbb{F}_q(H)=\mathbb{F}_q(X)[r]$,
and morphisms of degree two
\[   C\to H\to X.\]
Using these notations, our main result states the following.
\begin{theorem}\label{rank-formula}
 Let $E_{A,B}/\mathbb{F}_{q^2}(X)$ be defined by the
Weierstrass equation
\[
y^2+y = x^3+Ax^2+(A+1)x+A^2+B.
\] 
Assume that both $C/\mathbb{F}_{q^2}$ and $E/\mathbb{F}_{q^2}$
are maximal (or, both minimal). Then 
    \[ \textrm{rank } E_{A,B}(\mathbb{F}_{q^2}(X))=2g(C)-2g(H),\]
    where $g(C)$ and $g(H)$ denote the genera of the curves $C$ and $H$, respectively.
\end{theorem}

An ingredient in the proof (and similarly in other instances of the Tate-Shafarevich
approach) is the determination of $\textrm{rank}\,E(\mathbb{F}_{q^2}(D))$ for
a curve $D/\mathbb{F}_{q^2}$, assuming that both $D$ and $E$ are maximal over
$\mathbb{F}_{q^2}$ (resp., both minimal). Identifying $E(\mathbb{F}_{q^2}(D))$
with $\textrm{Mor}_{\mathbb{F}_{q^2}}(D,E)$, the Albanese property of
a jacobian yields an exact sequence
\[
0\to E(\mathbb{F}_{q^2})\longrightarrow \textrm{Mor}_{\mathbb{F}_{q^2}}(D,E)
\longrightarrow \textrm{Hom}_{\mathbb{F}_{q^2}}(\textrm{Jac}(D),E)\to 0.
\]
Hence $\textrm{rank}\,E(\mathbb{F}_{q^2}(D))=\textrm{rank}\,\textrm{Hom}_{\mathbb{F}_{q^2}}(\textrm{Jac}(D),E)$. The maximality condition on $D$ and $E$ implies
that the characteristic polynomials denoted $f_E$ and $f_{\textrm{Jac}(D)}$ in
Tate's paper \cite{Tate66} equal $(T+q)^2$ resp. $(T+q)^{2g(D)}$ (and $q$
replaced by $-q$ in the minimal case). Hence \cite[Thm~1(a)]{Tate66} implies
\[   \textrm{rank}\,E(\mathbb{F}_{q^2}(D)) = 4\cdot g(D).  \]

\section{The quartic twist}
In this section we explain the explicit equation for $E_{A,B}$ appearing in Theorem~\ref{rank-formula}.
Fix $q$ and curves $C,X$ over $\mathbb{F}_q$ and $\sigma\in\textrm{Aut}_{\mathbb{F}_q}(C)$ determined by
$A,B\in \mathbb{F}_q(X)$ as in Section~\ref{prelim}. Using moreover the elliptic curve $E$ and its
automorphism $\iota$, the projection $E\times C\to C$ induces a morphism 
\[ 
   (E\times C)/\langle \iota\times\sigma\rangle\to C/\langle \sigma\rangle = X.
\]
We will show that this defines an elliptic surface over $X$, in fact by establishing that its generic fiber
is the elliptic curve $E_{A,B}/\mathbb{F}_q(X)$. This is done by determining the subfield $\mathbb{F}$ of $\mathbb{F}_q(X)(x,y,r,s)$ of elements fixed by the action of the automorphism $\iota\times\sigma$. Since  \[\iota\times\sigma\colon\left(x,y,r,s\right)\mapsto\left(x+1,y+x,r+1,s+r\right),\] it is easy to see that $x+r, x^2+x$ belong to $\mathbb{F}$. We then have 
\[\mathbb{F}_q(X)(x+r,x^2+x)\subseteq\mathbb{F} \subseteq\mathbb{F}_q(X)(x,y,r,s).\] 
As a consequence $\iota\times\sigma$ yields an $\mathbb{F}_q(X)(x+r,x^2+x)$-linear map. Consider the chain of field extensions
\[\mathbb{F}_q(X)(x+r,x^2+x)\subseteq\mathbb{F}_q(X)(x,r)\subseteq\mathbb{F}_q(X)(x,y,r)\subseteq\mathbb{F}_q(X)(x,y,r,s).\] Each consecutive extension is of degree $2$ (because of the equations for $E$ and $C$). 
Hence $\mathbb{F}_q(X)(x,y,r,s)$ has degree $8$ and  basis $$\{1,x,y,xy,s,xs,ys,xys\}$$ over $\mathbb{F}_q(X)(x+r,x^2+x)$. In terms of this basis, the automorphism $\iota\times\sigma$ is represented by the $8\times8$ matrix

\[\begin{pmatrix}
    1&1&0&x^2+x&x+r&x^2+r&x^2+x&(x^2+x)(x+r)\\
    0&1&1&0&1&x+r&x+r+1&x^2+x\\
    0&0&1&1&0&0&x+r&x^2+r\\
    0&0&0&1&0&0&1&x+r\\
    0&0&0&0&1&1&0&x^2+x\\
    0&0&0&0&0&1&1&0\\
    0&0&0&0&0&0&1&1\\
    0&0&0&0&0&0&0&1
\end{pmatrix}.\]

The eigenspace associated to the eigenvalue $1$ (i.e. the elements fixed by the automorphism) consists of the vectors 
\[(a_1,(x+r)a_s,a_s,0,a_s,0,0,0).\] 
It follows that $[\mathbb{F}:\mathbb{F}_q(X)(x+r,x^2+x)]=2$ and a basis for $\mathbb{F}$ over $\mathbb{F}_q(X)(x+r,x^2+x)$ is given by $\{1,(x+r)x+y+s\}$. 
Let \[\eta=(x^2+r)\cdot 1+1\cdot((x+r)x+y+s)=y+rx+r+s\quad\textrm{and}\quad\xi=x+r.\] 
Since $\xi^2+\xi=x^2+r^2+x+r=x^2+x+A$, the elements $\xi$ and $\eta$ generate the field $\mathbb{F}$ over $\mathbb{F}_q(X)$.
One computes \[\eta^2=y^2+r^2(x^2+1)+s^2=y^2+(r+A)(x^2+1)+s^2\] and
\[\eta^2+\eta=x^3+(r+A)x^2+(r+1)x+(r+1)A+B.\] Since \[\xi^3=x^3+r x^2+r^2x+r^3=x^3+r x^2+(r+A)x+(r+1)A+r,\] it follows that 
\[\eta^2+\eta+\xi^3=Ax^2+(A+1)x+r=A(x^2+x)+x+r+B.\] Finally, the relation \[x^2+x=\xi^2+\xi+A\] results in 

\begin{equation}\label{twist-expression}
    \eta^2+\eta=\xi^3+A\xi^2+(A+1)\xi+A^2+B,
\end{equation}
which is the defining equation of $E_{A,B}$ over $\mathbb{F}_q(X)$. Note that the $j$-invariant of $E_{A,B}$ is $0$ (as is
obvious by construction: the degree $4$ base change $C\to X$ results in the isotrivial elliptic surface $E\times C\to C$
with generic fibre $E$ over $\mathbb{F}_q(C)$). Explicitly, by construction 

\begin{equation}\label{iso-E-EA}
    \begin{array}{ccc}
    E & \simeq & E_{A,B} \\
     (\xi+r,\eta+r\xi+A+s) & \mapsfrom & (\xi,\eta), \\
    (x,y) & \mapsto & (x+r,y+rx+r+s).
    \end{array}
\end{equation}
The discriminant of the Weierstrass equation for $E_{A,B}$ equals $1$. Hence only at the pole(s) of the functions $A,B$
the elliptic surface over $X$ may have singular fibers.

\section{The proof of the theorem}
The automorphism $\iota$ on $E$ induces an automorphism on $E_{A,B}$ which we again denote by $\iota$, namely 
\[ \iota\colon(\xi,\eta)\mapsto(\xi+1,\eta+\xi).\] 
Note that $\iota^2=[-1]$. Let $C$ be the curve discussed in Section~\ref{prelim} (and in particular in Theorem~\ref{rank-formula}). 
Then \[V=E_{A,B}(\mathbb{F}_{q^2}(C))\otimes\mathbb{Q}\] is in fact a $\mathbb{Q}(i)$-vector space, by setting 
\[i\cdot(P\otimes a)=\iota(P)\otimes a.\] 
The automorphism $\sigma$ given in equation \eqref{aut-sigma} induces a $\mathbb{Q}(i)$-linear map $\tau\colon V\to V$ as follows: write 
\[P=(x_0+x_1r+x_2s+x_3rs,\ y_0+y_1r+y_2s+y_3rs)\] with $x_j,y_j\in\mathbb{F}_q(X)$, then $\tau(P\otimes a)=$
\[\left(x_0\!+\!x_1\!+\!x_3A\!+\!(x_1+x_2)r\!+\!(x_2+x_3)s\!+\!x_3rs,\, 
y_0\!+\!y_1\!+\!y_3A\!+\!(y_1+y_2)r\!+\!(y_2+y_3)s\!+\!y_3rs\right)\otimes a.\] 
Note that $\mathbb{Q}(i)$-linearity follows from $\iota\circ\tau=\tau\circ\iota$ on $V$.

We will use $\tau$ to decompose $V$ as the sum of eigenspaces. Since $\tau^4=id_V$, the eigenvalues of $\tau$ are contained in $\{\pm1,\pm i\}$, hence 
\[V=\textrm{Ker}(\tau-id_V)\oplus\textrm{Ker}(\tau+id_V)\oplus\textrm{Ker}(\tau-i\cdot id_V)\oplus\textrm{Ker}(\tau+i\cdot id_V).
\]
We will now study each eigenspace.

\begin{itemize}
    \item $\textrm{Ker}(\tau-id_V)$: Comparing coordinates, one sees that $0\neq P\otimes a\in\textrm{Ker}(\tau-id_V)$ if, and only if, 
    \[P\otimes a=(x_0,\ y_0)\otimes a\] 
    with $(x_0,\ y_0)\in E_{A,B}(\mathbb{F}_{q^2}(X))$. Therefore, this eigenspace is precisely $E_{A,B}(\mathbb{F}_{q^2}(X))\otimes\mathbb{Q}$. 
    In particular, \[\dim_{\mathbb{Q}}\textrm{Ker}(\tau-id_V)=\textrm{rank}\,E_{A,B}(\mathbb{F}_{q^2}(X)).\]

    \item $\textrm{Ker}(\tau+id_V)$: Here $0\neq P\otimes a\in\textrm{Ker}(\tau+id_V)$ if, and only if, 
    \[P\otimes a=(x_0,\ y_0+r)\otimes a\] 
    with $(x_0,\ y_0+r)\in E_{A,B}$ and $x_0,y_0\in \mathbb{F}_{q^2}(X)$. From the isomorphism between $E$ and $E_{A,B}$ given in equation \eqref{iso-E-EA}, 
    it follows that those points correspond on $E$ to points of the form 
    \[(x_0+r,\ y_0+r+rx_0+A+s).\] 
    From \cite[Proposition A.1.2]{silverman}, we know that the automorphism group of $E$ consists of automorphisms 
    \[(x,\ y)\mapsto(u^2x+\alpha,\ y+u^2\alpha^2x+\beta),\] where 
    \[u^3=1,\quad \alpha^4+\alpha=u^2+1\quad \beta^2+\beta+\alpha^3+\alpha=0.\] Choose $u=1$, $\alpha\in\mathbb{F}_4\backslash\mathbb{F}_2$ and $\beta$ any root of $\beta^2+\beta=\alpha^3+\alpha=1+\alpha$ (which is an element of $\mathbb{F}_{16}\backslash\mathbb{F}_4$). The assumption that $E/\mathbb{F}_{q^2}$
    is either maximal or minimal, implies that the automorphism given by $u,\alpha,\beta$ is defined over $\mathbb{F}_{q^2}$. 
    Composing it with the isomorphism given in equation \eqref{iso-E-EA}, one obtains 
    \[(x_0+r,\ y_0+r+rx_0+A+s)\mapsto(x_0+\alpha, y_0+(\alpha+1)x_0+\beta).\] 
    The resulting points are invariant under the action of $\sigma$, so $\textrm{Ker}(\tau+id_V)$ can be identified with a subspace of $\textrm{Ker}(\tau-id_V)$ and therefore 
    \[\dim_{\mathbb{Q}}\textrm{Ker}(\tau+id_V)\leq\dim_{\mathbb{Q}}\textrm{Ker}(\tau-id_V).\] 
    On the other hand, starting from a point $(x_0,\ y_0)\otimes a\in\textrm{Ker}(\tau-id_V)$, the same procedure with the same isomorphisms shows 
    \[ \dim_{\mathbb{Q}}\textrm{Ker}(\tau-id_V)\leq\dim_{\mathbb{Q}}\textrm{Ker}(\tau+id_V)\] and so one concludes 
    \[\dim_{\mathbb{Q}}\textrm{Ker}(\tau+id_V)=\dim_{\mathbb{Q}}\textrm{Ker}(\tau-id_V)=\textrm{rank}\,E_{A,B}(\mathbb{F}_{q^2}(X)).\]

    \item $\textrm{Ker}(\tau-i\cdot id_V)$: Now, we have $P\otimes a$ in this eigenspace if, and only if, 
    \[P\otimes a=(x_0+r,\ y_0+x_0\cdot r+s)\otimes a\] 
    with $(x_0+r,\ y_0+x_0\cdot r+s)\in E_{A,B}$ and $x_0,y_0\in \mathbb{F}_{q^2}(X)$.  This translates into the
    condition that $(x_0,\ y_0)$ is an $\mathbb{F}_{q^2}(X)$-rational point on the elliptic curve 
    \[E_H\colon y^2+y=x^3+x+A.\] 
    As in Section~\ref{prelim}, let $H\to X$ be the curve corresponding to $r^2+r=A$. Over $\mathbb{F}_q(H)$ we have $E\simeq E_H$, in fact, an isomorphism  (and its inverse mapping) is given by $(x,\ y)\mapsto(x,\ y+r)$. In other words,
    $E_H$ is a quadratic twist of $E/\mathbb{F}_q(X)$, see, e.g., \cite{Elkies} and \cite[Section~3]{CST} for similar examples.
    The assumption that both $C$ and $E$ are maximal over $\mathbb{F}_{q^2}$ (or, both minimal) implies that the same
    holds for both $H$ and $E$. Hence as explained in Section~\ref{prelim},
   \[\textrm{rank}\,E_H(\mathbb{F}_{q^2}(H))=\textrm{rank}\,E(\mathbb{F}_{q^2}(H))=4 g(H).\]
   The curve $E_H$ being a quadratic twist of $E/\mathbb{F}_{q^2}(X)$ implies in particular that
   \[ \textrm{rank}\,E_H(\mathbb{F}_{q^2}(H))=\textrm{rank}\,E_H(\mathbb{F}_{q^2}(X))+\textrm{rank}\,E(\mathbb{F}_{q^2}(X)),\]
   as is shown by decomposing the group up to finite index into a $+1$ and a $-1$ eigenspace for the action of
   $\textrm{Gal}(\mathbb{F}_{q^2}(H)/\mathbb{F}_{q^2}(X))$. Reasoning as before yields
   \[\textrm{rank}\,E_H(\mathbb{F}_{q^2}(X))=4g(H)-\textrm{rank}\,E(\mathbb{F}_{q^2}(X))=4g(H)-4 g(X)\]
    leading to 
    the conclusion 
    \[\dim_{\mathbb{Q}}\textrm{Ker}(\tau-i\cdot id_V)=\textrm{rank}\,E_H(\mathbb{F}_{q^2}(X)) = 4\cdot g(H)-4g(X).\]

    \item $\textrm{Ker}(\tau+i\cdot id_V)$: The last eigenspace is generated by vectors $P\otimes a$ such that 
    \[P=(x_0+r,\ y_0+(x_0+1)r+s)\in E_{A,B}\] 
    and $x_0,y_0\in \mathbb{F}_{q^2}(X)$. Using the equation of $E_{A,B}$ this means that $(x_0,\ y_0)\in E(\mathbb{F}_{q^2}(X))$. Reasoning as before shows 
    \[\dim_{\mathbb{Q}}\,\textrm{Ker}(\tau+i\cdot id_V)=4g(X).\]
\end{itemize}

The maximality/minimality assumption on $C$ and $E$ implies 
$\dim_{\mathbb{Q}} V=4\cdot g(C)$, hence one concludes
\[ 4g(C)= 2\textrm{rank}\,E_{A,B}(\mathbb{F}_{q^2}(X)) + 4g(H)-4g(X)+4g(X)\]
finishing the proof of Theorem \ref{rank-formula}.

\section{Genera of the curves $C$ and $H$}\label{5:genera}
To make the result in Theorem~\ref{rank-formula} more explicit, we review some results on the genus of the curves
considered, in the special case that $X=\mathbb{P}^1$. Then $\mathbb{F}_q(X)$ equals the rational function field
$\mathbb{F}_q(t)$. We consider the case that $A=A(t), B=B(t)$ are polynomials in $t$. In particular this implies that
the elliptic surface over $\mathbb{P}^1$ defined by $E_{A,B}$ has only $1$ singular fiber (namely, at $t=\infty$).

The change of variables $(\tilde{r}=r+c, \tilde{s}=s+cr+d)$ with $c,d\in\mathbb{{F}}_q[t]$ brings the system
\[
\left\{\begin{array}{l}r^2+r=A(t)\\
s^2+s=rA(t)+B(t)\end{array}\right.
\]
in the form
\[
\left\{\begin{array}{l}\tilde{r}^2+\tilde{r}=\tilde{A}(t)\\
\tilde{s}^2+\tilde{s}=\tilde{r}\tilde{A}(t)+\tilde{B}(t),\end{array}\right.
\]
where $\tilde{A}(t)=A(t)+c^2+c$ and $\tilde{B}(t)=B(t)+(c^2+c)A(t) +c^3+c^2+d^2+d$.
As a consequence, for properties of the curves $C,H$ over $\mathbb{F}_q$, without loss of generality one may assume that $A(t)+A(0),B(t)+B(0)\in t\mathbb{\overline{F}}_q[t^2]$ (in particular, this results in polynomials of odd degree). We omit a proof of the following well-known result.
\begin{proposition}\label{genusH}
    If $A(t)\in\mathbb{F}_q[t]$ has odd degree, then the genus of the hyperelliptic curve $$H:r^2+r=A(t)$$ in characteristic $2$ is $g(H)=\frac{\deg A-1}{2}$. \hfill{$\Box$}
\end{proposition}
Considering the curve $C$, the result is as follows. For convenience
a sketch of the proof is added.
\begin{proposition}\label{prop: genus C}
    If $\deg A(t)$ and $\deg B(t)$ are odd, then the genus of the curve $C$ is $$g(C)=\max\left\{\frac{5\cdot\deg A-3}{2},\ \frac{2\cdot\deg B+\deg A-3}{2}\right\}.$$
\end{proposition}
\begin{proof}
    We use notations as in \cite[Proposition 3.7.8]{stichtenoth} and some Artin-Schreier theory for our wildly ramified field extensions $$\begin{array}{c}F'=\mathbb{\overline{F}}_q(t,r,s)\\\vline\\F=\mathbb{\overline{F}}_q(t,r)\\\vline\\K=\mathbb{\overline{F}}_q(t).\end{array}$$ There is only one place (called $P_{\infty}$ here) in $\mathbb{P}_F$ lying over $P_{1/t}$  and, above each finite place $P_{t-\alpha}\in\mathbb{P}_K$, where $\alpha\in\overline{\mathbb{F}}_q$, we have two places in $\mathbb{P}_F$, denoted $P_{\alpha}^0$ and $P_{\alpha}^1$. The rational function $w=A(t)$ viewed as an element of $\mathbb{\overline{F}}_q(t,r)$ has divisor $$\textrm{div}(w)=\sum m_{\alpha}P_{\alpha}^0 + \sum m_{\alpha}P_{\alpha}^1  - 2(\deg A)P_{\infty}.$$ Using $$\textrm{div}(r)=\sum m_{\alpha}P_{\alpha}^0-(\deg A)P_{\infty},$$ it follows that $$\textrm{div}(rA(t))=2\sum m_{\alpha}P_{\alpha}^0+\sum m_{\alpha}P_{\alpha}^1 -3(\deg A)P_{\infty}.$$

    The polynomial $B(t)$ has its only pole in $P_{\infty}$. Therefore, the only possible ramified place at the field extension $F'/F$ is $P_{\infty}$. We now need to find the smallest odd value of $v_{P_{\infty}}(rA(t)+B(t)+z^2+z)$, with $z\in F$. For that, choose a local parameter $u$ at $P_{\infty}$ and write $$r=u^{-\deg A}\cdot \eta,$$ with $v_{P_{\infty}}(\eta)=0$. Since $\deg A\not\equiv0$ (mod $2$), working in the completed local ring at $P_{\infty}$ one obtains $v$ such that $\eta=v^{-\deg A}$. Therefore, defining $\pi=u\cdot v$, one obtains $$r=\pi^{-\deg A}.$$

    Since $v_{P_{\infty}}(t)=-2$, write $$t=a_{-2}\pi^{-2}+a_{-1}\pi^{-1}+\cdots.$$ Let $$A(t)=\alpha_dt^d+\alpha_{d-1}t^{d-1}+\cdots+\alpha_0.$$ In terms of $\pi$, the equation $r^2+r=A(t)$ yields 
    \[
    \begin{array}{rl}
    \pi^{-2d}+\pi^{-d}= &\alpha_d(a_{-2}\pi^{-2}+a_{-1}\pi^{-1}+\cdots)^d+\alpha_{d-1}(a_{-2}\pi^{-2}+a_{-1}\pi^{-1}+\cdots)^{d-1}+\cdots+\alpha_0\\
    &\displaystyle{=\alpha_da_{-2}^d\pi^{-2d}+\alpha_da_{-2}^{d-1}a_{-1}\pi^{-2d+1}+\sum_{\alpha=2}^{\infty}c_{\alpha}\pi^{-2d+\alpha}},
    \end{array}
    \]
    where $$c_{\alpha}=\alpha_da_{-2}^{d-1}a_{\alpha-2}+\sum b_{\alpha,i}a_{-2}^{i_{-2}}a_{-1}^{i_{-1}}\cdots a_{\alpha-3}^{i_{\alpha-3}},\quad\textrm{with }(-2)i_{-2}+(-1)i_{-1}+\cdots+(\alpha-3)i_{\alpha-3}=-2d+\alpha,$$ where $b_{\alpha,i}$ involves the coefficients of $A(t)$ and binomial coefficients.

    This implies $$\alpha_da_{-2}^d=1,\quad a_{-1}=0,\quad\textrm{and}\;\; c_{\alpha}=0,\textrm{ if }\alpha\neq d.$$ The latter condition gives an expression of $a_{\alpha-2}$ in terms of $a_{-2},\ldots,a_{\alpha-3}$. If $\alpha$ is odd, we know that any monomial on the expression of $c_{\alpha}$ has at least one coefficient of odd index. Starting from $a_{-1}=0$ and using $c_{\alpha}=0$ if $\alpha\neq d$, we obtain, inductively, $$a_{-1}=a_1=\cdots=a_{d-4}=0,\quad \textrm{and}\quad \alpha_da_{-2}^{d-1}a_{d-2}=1.$$ Therefore, $$t=a_{-2}\pi^{-2}+a_0+\cdots+a_{d-2}\pi^{d-2}+\cdots.$$

    Writing $$B(t)=\beta_{d'}t^{d'}+\beta_{d'-1}t^{d'-1}+\cdots+\beta_0$$ (of odd degree $d'$) and using the expansion of $t$ obtained above, we conclude that the smallest odd power of $\pi$ appearing on $B(t)$ is $$\beta_{d'}a_{-2}^{d'-1}a_{d-2}\pi^{-(2d'-d)}.$$ Hence as above, one finds $z\in F$ such that $v_{P_{\infty}}(B(t)+z^2+z)=2\deg B-\deg A$. Then 
    \[
    \begin{array}{rl}
    -v_{P_{\infty}}(r\cdot A(t)+B(t)+z^2+z)&=\max\{v_{P_{\infty}}(r\cdot A(t)),\ v_{P_{\infty}}(B(t)+z^2+z)\}\\
    &=\max\{3\cdot\deg A,\ 2\cdot\deg B-\deg A\}=:m_{P_{\infty}}.
    \end{array}
    \]

    The genus formula for an Artin-Schreier extension now gives
    \[
    g(C)=2g(H)+\frac{m_{P_{\infty}}-1}{2}=\frac{2\deg A+m_{P_{\infty}}-3}{2}=\max\left\{\frac{5\deg A-3}{2},\ \frac{2\deg B+\deg A-3}{2}\right\}.
    \]
\end{proof}

%\begin{remark}
%    We can use the previous theorem to obtain the genus of $C$ when $\deg A(t)$ is even. However, this procedure depends on the coefficients of $A(t)$, so it is not possible to provide a formula for such genus. Proceeding as in Remark \ref{even-degree-remark}, we can find $z\in\mathbb{F}_q[t]$ such that $A(t)+z^2+z$ is a polynomial on $t$ of odd degree $d'$. Since $$(r+z)^2+(r+z)=A(t)+z^2+z,$$ we obtain $v_{P_{\infty}}(r+z)=-d'$ and, following the proof of the previous theorem, we find a local parameter $\pi$ at $P_{\infty}$ such that $$r+z=\pi^{-d'}\quad\textrm{and}\quad t=a_{-2}\pi^{-2}+a_0+\cdots+a_{d'-2}\pi^{d'-2}+\cdots$$ with $a_{-2},a_0,\ldots$ explicitly obtained in terms of the coefficients of $A(t)+z^2+z$. Since $z$ is a polynomial on $t$ with known expression, we use this information together with $$r+z=\pi^{-d'}$$ to obtain the expression of $r$ in terms of $\pi$. We are then able to find the expression of $r\cdot A(t)$ and $B(t)$ in terms of $\pi$ and, therefore, the genus of the curve $C$.
%\end{remark}

\section{examples}
Here some examples defined over $\mathbb{F}_q$ are presented where the conditions in Theorem~\ref{rank-formula} are satisfied. We always consider $\mathbb{F}_{16}\subseteq\mathbb{F}_q$. In particular, $q$ is a square.

\subsection{A genus $6$ example}
A simple and well-known case of a maximal curve over
$\mathbb{F}_{16}$ is the plane quintic
\[ C\colon y^4+y=x^5.\]
This curve has genus $(5-1)(5-2)/2=6$. An automorphism
$\sigma$ of order $4$ is given by
\[
\sigma\colon (x,y)\mapsto (x+c, y+c^4x+c)
\]
with $c\in\mathbb{F}_{16}$ satisfying $c^4+c^3+1=0$.
Then $c^5\in \mathbb{F}_4^\times$ has order $3$, and
\[
\sigma^2\colon (x,y)\mapsto (x,y+c^5).
\]
One observes that $\eta=c^5y^2+c^{10}y$ and $\xi=c^2x$
generate the field of invariants in the function field $\mathbb{F}_{16}(C)=\mathbb{F}_{16}(x,y)$ under $\sigma^2$, with relation
\[
\eta^2+\eta=\xi^5.
\]
The action of $\sigma$ on this subfield $\mathbb{F}_{16}(\xi,\eta)=\mathbb{F}_{16}(C/\langle \sigma^2\rangle)$ is given as
\[
\sigma\colon (\xi,\eta)\mapsto (\xi+c^3, \eta+c^9\xi^2+c^{12}\xi+c^{10}).
\]
and here convenient generators for the subfield of elements fixed by $\sigma$ acting on $\mathbb{F}_{16}(\xi,\eta)$ are
\[
u:=c^4\xi\cdot\sigma(\xi) \;\;\textrm{and}\;\;v:=c^{10}\eta\cdot\sigma(\eta).
\]
Replacing $u$ by $u_1=u+1$ and $v$ by $w:=u_1+(v+1)/u_1$ results in new generators,
 satisfying $w^2+w=u_1^3+u_1$, the equation of the elliptic curve
 curve $E$ considered throughout this paper.

To conclude: this example results in a case
\[
(E\times C)/\langle \iota\times\sigma\rangle\longrightarrow X=C/\langle\sigma\rangle = E
\]
over $\mathbb{F}_{16}$ with base curve $X=E$ of genus $1$, and with
Mordell-Weil rank \[2g(C)-2g(C/\langle\sigma^2\rangle)=12-4=8.\]

\subsection{Curves $C:y^2+y=x^{q+1}$}\label{6.2: Hermitian}

In this subsection we prove the following.
\begin{proposition}
    Let $q=2^{2n}$ with $n$ odd. Then, for any $b_0,c\in\mathbb{F}_{q^2}$ such that $$b_0^2+b_0=c^{q+1} \quad \textrm{and} \quad Tr_{\mathbb{F}_q/\mathbb{F}_2}(c^{q+1})=1,$$ the elliptic curve $$\eta^2+\eta=\xi^3+\left(\frac{t}{c^2}\right)\xi^2+\left(\frac{t+c^2}{c^2}\right)\xi+\frac{t^2}{c^4} + t\left(c^{q-1}+\left(b_0+\sum_{j=0}^{2n-2} \left(\frac{t}{c^2}\right)^{2^j}(b_0+1+b_0^{2^{j+1}})\right)^2/c\right)$$ has rank $q$ over $\mathbb{F}_{q^2}(t)$.
\end{proposition}
\begin{proof}
    Consider the curve $$C:y^2+y=x^{q+1}.$$ For convenience we recall the elementary
    argument showing that $C/\mathbb{F}_{q^2}$ is maximal. Note that Proposition~\ref{genusH}
    implies that the genus $g(C)$ of $C$ equals $q/2$. For any $\alpha\in\mathbb{F}_{q^2}$ the norm $\nu:=\alpha^{q+1}$ is in $\mathbb{F}_q$. Hence the separable (since the
    characteristic is $2$) polynomial
    $T^2+T+\nu$ has $2$ zeros $\beta\in\mathbb{F}_{q^2}$ and this yields $2q^2$ points
    $(\alpha,\beta)\in C(\mathbb{F}_{q^2})$. Together with the unique rational point at infinity one concludes
    \[\#C(\mathbb{F}_{q^2})=2q^2+1 = q^2+1+2g(C)q.
    \]
    Note that this works for any $q=2^m$. The condition $q=4^n$ and $n$ odd in the assumptions of the
    proposition assures that maximality/minimality of $C$ and of $E\colon y^2+y=x^3+x$ are aligned.
    
    We now  show that $\mathbb{F}_{q^2}(C)$ can be seen as a cyclic degree $4$ extension of $\mathbb{F}_{q^2}(t)$.
    
    It is known from \cite[Remark 3.3, Remark 4.2]{vG-vV} that the order $4$ automorphisms of $C$ are given by $$\varphi_c:(x,y)\mapsto(x+c,y+b_0+B(x)),$$ where $(c,b_0)\in C$, $B(x)=\sum_{i=0}^{2n-1}(c^qx)^{2^i}$ and $B(c)=1$. For any such $\varphi_c$, one has $$\varphi_c^2:(x,y)\mapsto(x,y+1)$$ so $\varphi_c^2$ is the hyperelliptic involution
    on $C$.
    
    Note that our assumption $Tr_{\mathbb{F}_q/\mathbb{F}_2}(c^{q+1})=1$ for $c\in \mathbb{F}_{q^2}$ is satisfied for precisely $q(q+1)/2$ elements $c$, as follows from observing
    the composition of surjective maps
    \[
    \mathbb{F}_{q^2}\stackrel{\textit{Norm}}{\longrightarrow}\mathbb{F}_q
    \stackrel{Tr_{\mathbb{F}_q/\mathbb{F}_2}}{\longrightarrow}\mathbb{F}_2.
    \]
    The condition on $c$ is equivalent to the statement that the point
    $(c,b_0)\in C(\mathbb{F}_{q^2})$ satisfies $b_0\in \mathbb{F}_{q^2}\setminus\mathbb{F}_q$.
    Fix such $(c,b_0)$.
By definition $B(c)=\sum_{i=0}^{2n-1}(c^qc)^{2^i}=Tr_{\mathbb{F}_q/\mathbb{F}_2}(c^{q+1})=1$.
    Hence $\varphi_c$ is an order $4$ automorphism defined over $\mathbb{F}_{q^2}$.  The field of invariants in $k(C)=k(x,y)$ under $\varphi_c^2$ is $k(x)$, and the action of $\varphi_c$ on $k(x)$ is given by $x\mapsto x+c$. As a result, $t:=x^2+cx$ generates
    the field of invariants under $\varphi_c$ in $\mathbb{F}_{q^2}(C)$. Denoting $r:=\frac{x}{c}$, one realizes $C$ as 
    $$C\colon \left\{\begin{array}{ll}
    r^2+r=\frac{t}{c^2}\\
    y^2+y=(c\cdot r)^{q+1}.
    \end{array}\right.$$
    
    Following the proof of \cite[Theorem 3]{AS} now brings $\mathbb{F}_{q^2}(C)$ in Artin and Schreier's standard form of a cyclic degree $4$ extension of $\mathbb{F}_{q^2}(t)$
    described on page~\pageref{ASform}. In fact, starting from $r^2=r+\frac{t}{c^2}$ and proceeding by induction shows that $$r^{2^i}=r+\frac{t}{c^2}+\left(\frac{t}{c^2}\right)^2+\cdots+\left(\frac{t}{c^2}\right)^{2^{i-1}}$$ and $$(c\cdot r)^{q+1}=c^{q-1}t+c^{q+1}\left(1+\left(\frac{t}{c^2}\right)+\left(\frac{t}{c^2}\right)^2+\cdots+\left(\frac{t}{c^2}\right)^{q/2}\right)r.$$ Using $c^{q+1}=b_0+b_0^2$, one writes 
    \[
    \begin{array}{rl}
    \varphi_c(y)&=y+b_0+B(c\cdot r)\\
    &=y+(b_0+b_0^q)r+\underbrace{b_0+\frac{t}{c^2}(b_0^2+b_0^q)+\left(\frac{t}{c^2}\right)^2(b_0^4+b_0^q)+\cdots+\left(\frac{t}{c^2}\right)^{q/4}(b_0^{q/2}+b_0^q)}_{D}.
    \end{array}
    \] 
    
    Since $\varphi_c^2(y)=y+1$, one concludes that $b_0+b_0^q=1$. Now, defining $s:=y+D\cdot r$, one has $\varphi_c(s)=s+r$ and
    $$\begin{array}{rl}
    s^2+s&=y^2+y+D^2\cdot r^2+D\cdot r\\
    &=c^{q-1}t+c^{q+1}\left(1+\left(\frac{t}{c^2}\right)+\left(\frac{t}{c^2}\right)^2+\cdots+\left(\frac{t}{c^2}\right)^{q/2}\right)r+D^2\left(r+\frac{t}{c}\right)+D\cdot r\\
    &=\underbrace{t\left(c^{q-1}+\frac{D^2}{c}\right)+\underbrace{\left[c^{q+1}\left(1+\left(\frac{t}{c^2}\right)+\left(\frac{t}{c^2}\right)^2+\cdots+\left(\frac{t}{c^2}\right)^{q/2}\right)+D^2+D\right]}_{E}r}_{\psi(r)}.
    \end{array}$$

    Finally, note that $$\psi(r+1)=\psi(\varphi_c(r))=\varphi_c(\psi(r))=\varphi_c(s^2+s)=(s+r)^2+(s+r)=\psi(r)+\frac{t}{c^2}.$$ Since $\psi$ is an expression of degree $1$ in $r$, $\psi(r+1)-\psi(r)$ is the leading coefficient, namely $E$, so one obtains $$s^2+s=r\cdot\left(\frac{t}{c^2}\right)+t\left(c^{q-1}+\frac{D^2}{c}\right).$$
    So an Artin-Schreier standard form for $C$ is
     $$C\colon \left\{\begin{array}{ll}
    r^2+r=\frac{t}{c^2}\\ \displaystyle
    s^2+s=r\frac{t}{c^2}+t\left(c^{q-1}+\left(b_0+\sum_{j=0}^{2n-2} \left(\frac{t}{c^2}\right)^{2^j}(b_0+1+b_0^{2^{j+1}})\right)^2/c\right).
    \end{array}\right.$$

    The proof of the proposition is now an immediate application of Theorem~\ref{rank-formula}.
 \end{proof}

 \begin{corollary}
     Given any natural number $M$, there is $q$ a power of $2$ and an elliptic curve defined over $\mathbb{F}_{q^2}(t)$ such that its rank over $\mathbb{F}_{q^2}(t)$ equals $q>M$.
 \end{corollary}

 \begin{remark}
     It is possible to replace the curve $y^2+y=x^{q+1}$ on this subsection by a more general curve $y^2+y=xR(x)$, where $R(x)$ is an additive polynomial, which is maximal (resp. minimal) over odd (resp. even) degree extensions of $\mathbb{F}_{16}$. This replacement will provide new families of elliptic curves defined over function fields of even characteristic having arbitrarily large Mordell-Weil rank. 
     
     This discussion will be presented in the second author's PhD thesis, as well as analogous results of this paper in the additional case of function fields of characteristic $3$ obtained from cubic twists of a supersingular elliptic curve defined over $\mathbb{F}_3$.
 \end{remark}

\subsection{Trace examples}\label{6.3: Herivelto}
For any positive integer $m$, put $Tr(t):=t+t^2+\cdots+t^{2^{m-1}}$ so
that $\alpha\mapsto Tr(\alpha)$ is the
trace function $\mathbb{F}_{2^m}\to\mathbb{F}_2$.
\begin{proposition}
    Let $q=2^m$ with $m$ odd. Then $$E_{Tr(t),0}:\eta^2+\eta=\xi^3+Tr(t)\xi^2+(Tr(t)+1)\xi+Tr(t)^2$$ has rank $q$ over $\mathbb{F}_{q^4}(t)$.
\end{proposition}
\begin{proof}
    Note that, for odd $m$, the curve
    \[
    H:r^2+r=Tr(t)=t+t^2+\cdots+t^{2^{m-1}}
    \]
    is rational; explicitly,  parametrized by
    \[
    (t,r)=(x^2+x,\ x+x^2+\cdots+x^{2^{m-1}}).
    \]
    Consider the Artin-Schreier covering of $H$ given by
    $$C:\left\{\begin{array}{rl}
    r^2+r&=Tr(t)\\
    s^2+s&=r\cdot Tr(t)
    \end{array}\right.$$
    Using the parametrization, the genus and the number of points of $C$ can be studied from the resulting equation 
    \[\begin{array}{ll}
    s^2+s=&(x+x^2+\cdots+x^{2^{m-1}})\cdot Tr(x^2+x)\\
    &=(x+x^2+\cdots+x^{2^{m-1}})\cdot(x^{2^m}+x)\\
    &=\sum_{n=0}^{m-1}\left( x\cdot x^{2^n}+x^{2^m}\cdot x^{2^n}\right).
    \end{array}
    \]
    Since $x^{2^n}\cdot x^{2^m}\equiv x\cdot x^{2^{m-n}}\bmod \wp(\mathbb{F}_2[x])$ with as before $\wp$ the Artin-Schreier operator 
    $f\mapsto f^2+f$, one may replace $x^{2^n}\cdot x^{2^m}$ by $x\cdot x^{2^{m-n}}$ in the latter expression. As a consequence, the curve $C$ is isomorphic to the one with equation 
    $$\begin{array}{rl}
    s^2+s&=(x\cdot x^{2^m}+x\cdot x^{2^{m-1}}+\cdots+x\cdot x^2)+x(x+x^2+\cdots+x^{2^{m-1}})\\
    &=x(x^{2^m}+x).
    \end{array}$$
    Hence the genus of $C$ equals $2^{m-1}={q}/{2}$. Maximality over $\mathbb{F}_{q^4}$ follows from the results of \cite{coulter}. Theorem~\ref{rank-formula} now finishes the proof.
\end{proof}

\vspace{\baselineskip}
Observe that the examples in Subsections
\ref{6.2: Hermitian} and \ref{6.3: Herivelto} involve only hyperelliptic curves. We have not found any nonhyperelliptic curve
 $C$ with $\sigma\in\textrm{Aut}_{\mathbb{F}_{q^2}}(C)$ for large $q$ satisfying both maximality over $\mathbb{F}_{q^2}$ and
 $C/\langle\sigma\rangle\cong\mathbb{P}^1$. Since a potential example
 satisfies
 \[ q^2+1+2g(C)q=\#C(\mathbb{F}_{q^2})\leq 4\cdot\#\mathbb{P}^1(\mathbb{F}_{q^2})\]
 it would necessarily have $g(C)<3(q^2+1)/(2q)$.
\section{The Mordell-Weil rank via the Shioda-Tate formula}

In what follows we determine the geometric Mordell-Weil rank of the elliptic curve $E_{A,B}$, namely its rank over $\overline{\mathbb{F}}_q(t)$. This is a direct application of the Shioda-Tate formula and hence depends only on the type(s) of singular fiber(s) of $E_{A,B}$ and the Picard number of the associated elliptic surface. The former depends only on the degrees of $A(t)$ and $B(t)$ and is described on Prop. \ref{prop:singularfibers}, while the latter is equal to the second Betti number, since $E_{A,B}$ is supersingular. We conclude the section by comparing the geometric Mordell-Weil rank with that over $\mathbb{F}_{q^2}$ obtained in Section \ref{prelim}.

\begin{proposition}\label{prop:singularfibers}
    Let $A,B\in\mathbb{F}_q[t]$ be polynomials of odd degree. Consider the elliptic surface $E_{A,B}$ given in \eqref{twist-expression}. The only singular fiber of this surface is at infinity, and its type can be found on the following table.

    \begin{center}
    \begin{tabular}{c|c}
    Condition on the degrees & Fiber type \\
    \hline
    $3\deg A(t)\geq\deg B(t)$ & $I_m^*$, $m=2\cdot\min\{\deg A(t), 3\deg A(t)-\deg B(t)\}$\\
    $3\deg A(t)<\deg B(t)\equiv 1$ (mod $6$) & $II^*$ \\
    $3\deg A(t)<\deg B(t)\equiv 3$ (mod $6$) & $I_0^*$ \\
    $3\deg A(t)<\deg B(t)\equiv 5$ (mod $6$) & $II$
    \end{tabular}
    \end{center}
\end{proposition}
\begin{proof}
   Consider $$E_{A,B}: \eta^2+\eta=\xi^3+A(t)\xi^2+(A(t)+1)\xi+\underbrace{A(t)^2+B(t)}_{a_6(t)}.$$ The discriminant of this equation is $\Delta\equiv 1$. Therefore, the only possible singular fiber is at infinity. To study its behavior, we denote by $n$ the Euler number of the associated elliptic surface, i.e., the smallest integer such that 
    $$\deg A(t)\leq 2n \quad \textrm{and} \quad \deg(A(t)^2+B(t))\leq6n.$$ More precisely, we have 
    $$n=\left\{\begin{array}{ll}
    \frac{\deg A(t)+1}{2}&, \textrm{if }\deg B(t)\leq3\cdot\deg A(t)\\
    \frac{\deg B(t)+m}{6}&, \textrm{if }\deg B(t)>3\cdot\deg A(t) \textrm{ and } \deg B(t)\equiv -m\ (\textrm{mod }6).
    \end{array}\right.$$
    
    Let $$x=\frac{\xi}{t^{2n}},\quad y=\frac{\eta}{t^{3n}}\quad \textrm{and}\quad s=\frac{1}{t},$$ and divide the equation of $E_{A,B}$ by $t^{6n}$. We obtain
    $$y^2+s^{3n}y=x^3+s^{2n-\deg A(t)}\tilde{A}(s)x^2+s^{4n-\deg A(t)}\overline{A}(s)x+ s^{6n-\deg a_6(t)}\tilde{a}_6(s),$$ where $$s\nmid \tilde{A}(s), \overline{A}(s), \tilde{a}_6(s).$$

Then the quantities as in \cite[Sec. III.1]{silverman} are as follows:
 $$b_2=b_4=0,\quad b_6=s^{6n},\quad b_8=s^{8n-\deg A(t)}\tilde{A}(s)+s^{8n-2\deg A(t)}\overline{A}(s).$$
    This puts one in position to apply Tate's algorithm \cite{tate} to find the type of the fiber at infinity (which is now the fiber at $s=0$). Since $b_2=0$, the singular fiber is of additive type. If $3\deg A(t) \geq \deg B(t)$ then, following Tate's algorithm, one shows that the fiber is of type $I_m^*$. Otherwise, the specific type depends on the congruence class of the degree of $B(t)$ modulo 6.

\end{proof}

\begin{remark}
The phenomenon described in Proposition \ref{prop:singularfibers}, namely the presence of a unique singular fiber in a nontrivial elliptic fibration, is particular to fields of characteristic 2 and 3 and described, for instance, in \cite{WLang} for rational elliptic with Mordell-Weil rank zero. More surprisingly, Proposition \ref{prop:singularfibers} shows that there are examples of arbitrarily high Mordell-Weil rank such that all fibers are irreducible and only one of them is singular. Indeed, the ones given by of the form $E_{A,B}$ as in \eqref{twist-expression} with $3\deg A(t)< \deg B(t)$ and $\deg B(t) \equiv 5 \mod 6$ are of such type.
\end{remark}

As a corollary of the proof of Proposition \ref{prop:singularfibers} we obtain the following.

\begin{corollary}
   The second Betti number of the elliptic surface $E_{A,B}$ is 
   \begin{equation}
   b_2= \left\{ \begin{array}{cc}
       12 \lceil \frac{\deg B(t)}{6}\rceil -2  & ,\  \text{ if } \deg B(t) > 3\deg A(t)\\
       12 \lceil \frac{\deg A(t)}{2}\rceil -2 &, \ \text{ if } \deg B(t) \leq 3\deg A(t). 
   \end{array} \right.
   \end{equation}

%\begin{itemize}
%\item[i)] The curve $E_{A,B}$ is super-singular;

%\item[iv)] 
%\end{itemize}
\end{corollary} 
\begin{proof}
%Item (i) is well-known. The reader can consult \cite[Prop. 4.31]{Washington}.
%To prove (ii) it is enough to notice that the discriminant of $E_{A,B}$ is identically 1. Hence the unique place of bad reduction is at $t=\infty$. From the latter, we conclude that the fiber at $t=\infty$ is additive. Item (iii) then follows from the fact that $\deg B(t) \equiv 3 \mod 6$.
This follows from \cite[Sec. 6.10]{ShiodaSchuett} for elliptic fibrations with rational base combined with the proof of Prop. \ref{prop:singularfibers}. 

\end{proof}

\begin{theorem}\label{thm: geometric rank}
Let $r$ be the geometric Mordell-Weil rank of $E_{A,B}$. Then 
\begin{equation}
r = -2+ \left\{ \begin{array}{cc}
       2\deg B(t) &  \text{, if } \deg B(t) > 3\deg A(t)\\
       6 \deg A(t) -2\cdot\min\{\deg A(t), 3\deg A(t)-\deg B(t)\}  &\text{, if } \deg B(t) \leq 3\deg A(t).
\end{array}
\right.
\end{equation}

\end{theorem}
\begin{proof}
Since $E_{A,B}$ is supersingular, the rank of the N\'eron-Severi group $\rho$ of the elliptic surface associated to $E_{A,B}$ is equal to the second Betti number of the surface, $b_2$. This combined with the information about the singular fiber of $E_{A,B}$ given in Prop. \ref{prop:singularfibers} puts us in position to apply the Shioda-Tate formula (\cite[Cor. 6.13]{ShiodaSchuett}) and obtain the desired equality for $r$.
\end{proof}

%\begin{corollary}
 %   If $E_{A,B}$ is supersingular then the equality hold in Theorem \ref{thm: geometric rank}.
%\end{corollary}

%If both curves $C$ and $E$, as in Section \ref{prelim}, are maximal over $\mathbb{F}_{q^2}$ (or both curves are minimal), then the elliptic surface $E_{A,B}$ will be supersingular. Using the genus formula for the curve $C$, we can then see that the rank formula obtained on our main theorem and the one given by the Shioda-Tate formula coincide.
We can finally interpret the results of this section in the light of the previous ones.

\begin{corollary}
    If the curves $C$ and $E$, as in Section \ref{prelim}, are both maximal or minimal over $\mathbb{F}_{q^2}$ then the geometric Mordell-Weil rank of $E_{A,B}$ is attained over $\mathbb{F}_{q^2}$.
\end{corollary}
\begin{proof}
This is a combination of Theorem \ref{thm: geometric rank}, Theorem \ref{rank-formula} and Proposition \ref{prop: genus C}.
\end{proof}

\section{Acknowledgments}
This work was developed during the stay of the second author at the Bernoulli Institute of the Rijksuniversiteit Groningen. He is happy to thank the institute for its hospitality. 

Herivelto Borges was supported by CNPq (Brazil), grants 406377/2021-9 and 310194/2022-9, and by FAPESP (Brazil), grant 2022/15301-5. João Paulo Guardieiro was supported by CNPq (Brazil), grant 140589/2021-0, and by CAPES (Brazil), finance code 001. Cec\'ilia Salgado was partially supported by the NWO-XL \textit{Rational Points: new dimensions} consortium grant. 

\bibliographystyle{plain}
\bibliography{refs}
\end{document}